\documentclass[12pt]{article}
\usepackage{mathtools,amsfonts,amssymb,amsthm,amscd}
\usepackage{newpxtext,newpxmath}
\usepackage{mathrsfs}
\usepackage{color}
\usepackage[all]{xy}
\usepackage{stmaryrd}
\newcommand{\field}[1]{\mathbb{#1}}

\newcommand{\Q}{\field{Q}}

\newcommand{\gm}{\field{G}_m}

\newcommand{\qst}[1]{R_{#1} \gm}
\newtheorem{theorem}{Theorem}[section]
\newtheorem{lemma}[theorem]{Lemma}
\newtheorem{proposition}[theorem]{Proposition}
\newtheorem{corollary}[theorem]{Corollary}

\newtheorem*{assumption}{Assumption}
\theoremstyle{definition}
\newtheorem{example}[theorem]{Example}
\newtheorem{definition}[theorem]{Definition}

\theoremstyle{remark}

\DeclareMathOperator{\gal}{Gal}

\allowdisplaybreaks
\numberwithin{equation}{section}
\begin{document}
\title{Linear recurrent sequences providing decomposition law in number fields}
\author{Haruto Hori and Masanari Kida\thanks{This work was supported by JSPS KAKENHI Grant Number 20K03521.}}
\date{\today}
\maketitle
\begin{abstract}
In their recent paper, Rosen, Takeyama, Tasaka, and Yamamoto constructed recurrent sequences 
that provide a decomposition law of primes in a Galois extension.
In this paper, we reconstruct their sequences via the representation theory 
of finite groups and obtain an explicit description of the sequences. 
\end{abstract}

\section{Introduction} \label{sec:1}
Let $L/\Q$ be a Galois extension of degree $d$ with Galois group 
$G=\gal (L/\Q)$ defined by the monic 
polynomial
\begin{equation} \label{eq:1.1}
 F(X) = c_0 + c_1 X + \cdots + c_{d-1}X^{d-1} + X^d
\end{equation}
with integer coefficients. 
In their paper \cite{MR4755038}, Rosen, Takayama, Tasaka, and Yamamoto
 constructed linear recurrent 
rational sequences $(a_i)_{i \ge 0}$ 
whose characteristic polynomial 
coincides with \eqref{eq:1.1} providing a decomposition law of a prime $p$ in 
$L$. More precisely, for each conjugacy class $K$ of $G$, 
they constructed such a sequence $(a_i)_{i \ge 0}$ 
satisfying 
\begin{equation} \label{eq:1.2}
 a_p \equiv \begin{cases}
	     1 \pmod{p} & \text{ if $\mathrm{Frob}_p \in K$}, \\
	     0 \pmod{p} &  \text{ otherwise}
	    \end{cases}
\end{equation}
for all but finitely many primes $p$. Here we denote by $\mathrm{Frob}_p$ the 
Frobenius automorphism of $p$.

Although their construction relies on studying certain Galois modules associated 
to the recurrent sequence, it can be clarified by using 
the representation theory of finite group and natural multiplicative structure 
of the modules.
The paper aims to explain this relation. 
This approach also enables explicit computation of the sequences.

The following assumption is adopted throughout the paper.

\begin{assumption}
A root $\xi $ of $F(X)$ generates 
a normal basis of $L$:
\[
 L = \bigoplus_{\sigma \in \gal (L/\Q)} \Q \sigma \xi .
\] 
\end{assumption}

Under this assumption, our main theorems are:

\begin{theorem} \label{thm:1.1}
Let $K_j \ (1 \le j \le r)$ be the conjugacy classes of $G$.
Let $(a_{K_j,i})_{i \ge 0}$ be the linear recurrent sequence 
whose characteristic polynomial is $F(X)$ in \eqref{eq:1.1}
satisfying \eqref{eq:1.2} for the conjugacy class $K_j$. 
The first $d$ terms of $(a_{K_j,i})_{i \ge 0}$
are given by the $j$-th column of the matrix
\[
 [\sigma^{-1}\xi^{i-1}]_{1 \le i \le d, \sigma \in G} [\sigma \tau^{-1} \xi]_{\sigma \in G, \tau
\in G}^{-1} [\kappa_{\tau,j}]_{\tau \in G, 1\le j \le r} 
\]
where 
\[
 \kappa_{\tau , j } = \begin{cases}
			      1 & \text{ if } \tau \in K_j, \\
			      0 & \text{ otherwise}.
			     \end{cases}
\]
\end{theorem}

By expanding the inverse matrix 
 $[\sigma \tau^{-1} \xi]^{-1}$, we also obtain an explicit expression of general
terms of the sequences.

\begin{theorem} \label{thm:1.2}
Let $X_{\sigma }$ be the variables indexed by $G$. We define 
$\Gamma = [X_{\sigma \tau^{-1}}]_{\sigma \in G, \  \tau \in G} $
and $\partial_{\tau} =\dfrac{1}{|G|} 
\dfrac{\partial \det \Gamma }{\partial X_{\tau}} $.
We evaluate $X_{\sigma}$  by $ \sigma \xi $
and we denote the resulted values by $\Gamma_{\xi }$ and 
$\partial_{\tau} (\xi)$.
Then, for any integer $i \ge 0$, 
we have
 \[
  a_{K_j, i}= \frac{1}{\det \Gamma_{\xi}} 
\sum_{\substack{\sigma \in G \\ \tau \in K_j }} \sigma \xi^{i}  \partial_{\sigma \tau } (\xi ) .
 \]
Moreover, when we regard $G$ as a transitive subgroup of the symmetric group $S_d$, 
if $G $ is not contained in the alternating subgroup $ A_{d}$, then we set $G^+ = G \cap A_d$ and we choose an element $\delta \in G$ such that $G=G^+ \sqcup \delta G^{+}$.
Then the following formulas are derived:
\[
 a_{K_j, i}= \begin{cases}
\displaystyle  \frac{1}{\det \Gamma_{\xi}} \mathrm{Tr}_{G}  \left( \xi^{i}
 \sum_{\tau \in K_j}  \partial_{\tau} (\xi ) \right) & \text{ if $G \subset A_{d}$},                         \\
\displaystyle	       \frac{1}{\det \Gamma_{\xi}} (\mathrm{Tr}_{G^{+}} -\delta \mathrm{Tr}_{G^{+}}) \left( \xi^{i}
 \sum_{\tau \in K_j}  \partial_{\tau} (\xi )\right) & \text{ if $G \not\subset A_{d}$}. 
	     \end{cases}
\]
\end{theorem}

By the formulas in Theorem \ref{thm:1.2}, we can compute 
each term of the sequences theoretically without 
using the recurrence formula.

The outline of the paper is as follows. In Section \ref{sec:2}, we reformulate
the proof in \cite{MR4755038} on the existence of the 
sequence satisfying \eqref{eq:1.2}.
This section concludes with our strategy 
of the proof of Theorem \ref{thm:1.1}, which will be given in Section \ref{sec:3}.
We prove
 the formulas in Theorem \ref{thm:1.2} in Section \ref{sec:4}.
From these formulas, we can directly deduce the rationality of the sequence 
(Corollary \ref{cor:4.5}).
In Section \ref{sec:5}, we give some specific examples of our theorems.
In the final section of this paper, we provide a clue to determine the finite set
of exceptional primes in \eqref{eq:1.2}.

We use the following notation throughout the paper.
For a finite group $G$, let $K_1, \ldots , K_r$ be the conjugacy classes of 
$G$ with representatives $\rho_i \ (1 \le i \le r)$. 
When we regard $G$ as an ordered set, the order is chosen so that 
it is compatible with the decomposition 
\[
 G = \bigsqcup_{i=1}^r K_i
\]
by furnishing arbitrary order on $K_i$. We fix this order once and for all.
For an element $\rho \in G$, the centralizer of $\rho$ is denoted by 
$Z_G (\rho)$.
We denote by $\mathrm{Irr} (G)$ the set of the irreducible characters of $G$ 
and, for $\psi \in \mathrm{Irr} (G)$,
by $\Lambda_{\psi}$ a complex irreducible representation
affording the character $\psi$. We also fix an order on the set $\mathrm{Irr} (G)$.

\section{Structure of the proof}  \label{sec:2}
From Section \ref{sec:1}, we recall that $L/\Q$ is
 a Galois extension of degree $d$ 
defined by the monic irreducible polynomial $F(X)$ in
\eqref{eq:1.1} with integer coefficients.
For short, we have set $G = \gal (L/\Q)$.

Let $\mathrm{RS}(F,\Q)$ be the $d$-dimensional $\Q$-vector space of 
rational recurrent sequences whose characteristic polynomial is equal to $F (X)$.
In this section, we analyze the proof of 
the existence of the sequences satisfying \eqref{eq:1.2} by Rosen, Takeyama, 
Tasaka and Yamamoto in \cite{MR4755038} from a constructive point of view.
 In their paper \cite{MR4755038}, they use a
map from $\mathrm{RS} (F,\Q)$ to so-called the ring of finite algebraic numbers.
This map eliminates the ambiguity of a finite number of exceptional primes, but 
does not relate to the construction of the sequences \eqref{eq:1.2}.
Thus we do not use the ring in the following construction and 
will not mention it further in the later sections.

We define the sequences $\mathbf{e}_i = (e_i^{(j)})_{j \ge 0} \in \mathrm{RS} (F,\Q ) \ (i=0,\ldots ,d-1) $ to satisfy
 $e_{i}^{(j)} = \delta_{ij}$ for $0 \le j \le d-1$.
The ordered 
set $(\mathbf{e}_i)_{0 \le i \le d-1}$ is a $\Q$-basis of $\mathrm{RS} (F,\Q )$.
We consider the linear transformation 
\[
 \nu : (a_0,a_1, \ldots ) 
\mapsto (a_1, a_2, \ldots )
\]
on $\mathrm{RS} (F,\Q)$. The matrix 
representation of the map $\nu$ with respect to the above basis is given by 
the companion matrix of $F(X)$:
\[
   R_F = \begin{bmatrix}
	  0 & 1 & 0 & \cdots  & 0 \\
	  0 & 0 & 1 & &  \\
	  \vdots &\vdots & &\ddots & \vdots \\
         0	& 0 & & & 1 \\
	 -c_0 & -c_1 & \cdots & & -c_{d-1}
	 \end{bmatrix}.
\]
Since $F(X)$ splits in $L$ and is separable, the matrix $R_F$ is diagonalizable
by 
\begin{equation} \label{eq:2.0}
 P=[\sigma^{-1} \xi^{i-1}]_{ 1 \le i \le d , \sigma \in G} 
\in \mathrm{Mat}_{d \times d} (L ) 
\end{equation}
and we indeed obtain 
\begin{equation} \label{eq:2.0.5}
 P^{-1} R_F P = \mathrm{diag} [\sigma^{-1}\xi]_{\sigma \in G}.
\end{equation}
This also shows that 
$\mathrm{RS} (F,L ) =L \otimes_{\Q}
\mathrm{RS} (F,\Q )$ has an $L$-basis consisting of  
the geometric series 
 $x_{\sigma^{-1} }= ((\sigma^{-1} \xi)^i)_{i \ge 0}$.

\begin{lemma} \label{lem:2.1}
These two bases of $\mathrm{RS} (F,L )$ relate by
\[
 (\mathbf{e}_0, \ldots , \mathbf{e}_{d-1}) = (x_{\sigma^{-1}})_{\sigma \in G} P^{-1}
\]
with $P$ as in \eqref{eq:2.0} and the matrix representation of $\nu$
with respect to the basis $(x_{\sigma^{-1}})_{\sigma \in G}$ is given by 
\[
 (\nu (x_{\sigma^{-1}}))_{\sigma \in G} = (x_{\sigma^{-1}})_{\sigma \in G} \,
\mathrm{diag}[\sigma^{-1} \xi ] _{\sigma \in G}.
\]
\end{lemma}
\begin{proof}
The definition of $P$ implies
\[
 (\mathbf{e}_0, \ldots , \mathbf{e}_{d-1}) P = (x_{\sigma^{-1}})_{\sigma \in G}.
\]
Since $P$ is a matrix of the change of the bases, it is invertible, and thus 
we have the first equality.

We consider a $\Q$-linear isomorphism
\[
 \Lambda : \mathrm{RS} (F,\Q ) \mapsto L, \quad 
\mathbf{e}_{i-1} \mapsto \xi^{i-1}.
\]
This map makes the following diagram commutative:
\[
 \xymatrix{
\mathrm{RS} (F,\mathbb{Q}) \ar[r]^{\ \ \Lambda} \ar[d]^{\nu} & L \ar[d]^{\times \xi} \\
\mathrm{RS} (F,\mathbb{Q}) \ar[r]^{\ \ \Lambda} & L
}
\]
where the right vertical map is the multiplication-by-$\xi$ map and the corresponding 
matrix with respect 
to the basis $(\xi^{i-1})_{1 \le i \le d}$ coincides with  $R_F$.
The first equality yields 
the commutative diagram of matrix multiplications:
\[
 \xymatrix{
L \otimes_{\Q} \mathrm{RS} (F,\mathbb{Q}) \ar[r]^{\ \ P^{-1}} \ar[d]^{1 \otimes R_F} & 
\mathrm{RS} (F,L)  \ar[d]^{\nu_L}   \\
L \otimes_{\Q} \mathrm{RS} (F,\Q ) \ar[r]^{\ \ P^{-1}} & \mathrm{RS}(F,L) .
}
\]
Therefore 
the right vertical induced map $\nu_L$  is represented by 
\[
 P^{-1} R_F P = 
\mathrm{diag} [\sigma^{-1} \xi]_{\sigma \in G}
\]
 with respect to the basis 
$(x_{\sigma^{-1}})_{\sigma \in G}$. This completes the proof of Lemma \ref{lem:2.1}.
\end{proof}

For a sequence $\mathbf{a}=(a_i)_{i \ge 0} \in \mathrm{RS} (F, L ) $, it follows from 
Lemma \ref{lem:2.1} that the coordinates of $\mathbf{a}$
with respect to $(x_{\sigma^{-1}})_{\sigma \in G}$ is given by
\begin{equation} \label{eq:co}
\begin{bmatrix}
 z_{\sigma}
\end{bmatrix}_{\sigma \in G}=
 P^{-1} \begin{bmatrix}
	 a_0 \\ \vdots \\ a_{d-1}
	\end{bmatrix} .
\end{equation}
This implies the equalities
\begin{align}
 & \mathbf{a}  = (\mathbf{e}_0, \ldots, \mathbf{e}_{d-1}) \begin{bmatrix}
						 a_0 \\ \vdots \\ a_{d-1}
						\end{bmatrix}
= (x_{\sigma^{-1}})_{\sigma \in G} \begin{bmatrix}
						 z_{\sigma}
						\end{bmatrix}_{\sigma \in G}
 \label{eq:2.2.7}\\ 
\intertext{and} 
& \begin{bmatrix}
	 a_n \\ \vdots \\ a_{n+d-1}
	\end{bmatrix}
= R_F^n \begin{bmatrix}
	 a_0 \\ \vdots \\ a_{d-1}
	\end{bmatrix} 
= P \mathrm{diag } [(\sigma^{-1}\xi)^n]  
\begin{bmatrix}
						 z_{\sigma}
						\end{bmatrix}_{\sigma \in G}. \label{eq:2.3}
\end{align}

Now let us define an $L$-linear map 
 $\Psi : \mathrm{RS} (F ,L)  \rightarrow L \otimes_{\Q } L$ 
by 
\begin{equation} \label{eq:2.2.5}
\Psi (\mathbf{a}) = \sum_{\sigma \in G} z_{\sigma} \otimes \sigma^{-1} \xi
\end{equation}
in accordance with \eqref{eq:co}.
The vector space $\mathrm{RS} (F ,L)$ has a natural termwise action of $G$ and
also $L \otimes_{\Q} L$ has an $L$-structure given by 
by $\eta \mapsto \eta \otimes 1 $
and has a diagonal $G$-action.

As in \cite[Proposition 2.5]{MR4755038}, we can prove the following lemma.

\begin{lemma} \label{lem:2.2}
The $L $-linear map $\Psi$ 
is a $G$-equivariant map.
\end{lemma}

Let $X(G, L )$ be 
the set of functions from $G$ to $L$.
Pointwise addition and multiplication endow $X(G,L)$ an $L$-algebra structure.
For $\varphi \in  X(G, L ) $ and $\tau \in G$, we define the left $G$-action 
on $X(G, L) $ by
\[
 ( \tau \varphi )( \sigma ) = \tau \varphi (\tau^{-1} \sigma \tau ).
\]

The following lemma can also be verified as in \cite[Section 2.3]{MR4755038}.

\begin{lemma} \label{lem:2.3}
 Let 
 $\Phi : L \otimes_{\Q} L \longrightarrow X (G , L)$ 
be the map defined by 
\[
\gamma \otimes \alpha  \mapsto (\tau \mapsto \gamma \tau \alpha).
\]
Then the map $\Phi$ induces a $G$-equivariant isomorphism 
\[
 L \otimes_{\Q} L \cong  X(G,L).
\]\
\end{lemma}

So far, we have $G$-equivariant homomorphisms
\begin{equation} \label{eq:2.5}
 \mathrm{RS} (F,L)  \overset{\Psi}{\longrightarrow}  L \otimes_{\Q } L 
\overset{\Phi}{\longrightarrow} X(G,L) .
\end{equation}
By taking the $G$-invariant part of the sequence, we have the following
sequence of $\Q$-linear maps as in \cite[(2.3)]{MR4755038}:
\begin{equation} \label{eq:2.8}
  \mathrm{RS} (F,\Q)  \overset{\Psi}{\longrightarrow}  (L \otimes_{\Q } L )^G
\overset{\Phi}{\longrightarrow} X(G,L)^G .
\end{equation}
Here by the definition of the $G$-action on $X(G,L)$, we see
\[
 X(G,L)^G = \{ \varphi \in X(G,L) \mid \tau \varphi (\sigma ) = 
\varphi (\tau \sigma \tau^{-1})\}.
\]

Let $p$ be a prime number and $\mathfrak{p}$ a prime ideal of $L$ lying above $p$.
If $p$ in unramified in $L$, then we denote by $\mathrm{Frob}_{\mathfrak{p}}$ the 
Frobenius automorphism of $\mathfrak{p}$.

Let $\phi_K \in X(G,\Q )$ be the characteristic 
function of 
a conjugacy class $K$ of $G$, that is, $\phi_K (\sigma) =1 $ if $\sigma \in K$,
and $\phi_K (\sigma)=0$ otherwise. Since $\phi_K$ is a class function, it belongs to 
$X(G,L)^G$.
If there is a sequence $\mathbf{a} \in \mathrm{RS} (F,\Q)$ 
satisfying $\Phi \circ \Psi (\mathbf{a})=
\phi_K $,
 then we have 
\begin{align*}
 \phi_K (\mathrm{Frob}_{\mathfrak{p}}) & = \left( \Phi \circ \Psi (\mathbf{a})\right) 
(\mathrm{Frob}_{\mathfrak{p}} ) \\
&=  \Phi \left( \sum_{\sigma \in G} z_{\sigma} \otimes \sigma^{-1} \xi \right)(\mathrm{Frob}_{\mathfrak{p}} ) && \text{ by } \eqref{eq:2.2.5}\\
&  = \sum_{\sigma \in G} z_{\sigma} \otimes \mathrm{Frob}_{\mathfrak{p}} ( \sigma^{-1} \xi) && \text{ by }  \text{Lemma } \ref{lem:2.3} \\ 
& \equiv \sum_{\sigma \in G} z_{\sigma} \otimes (\sigma^{-1} \xi )^p
   \pmod{\mathfrak{p}} \\
 & = \begin{bmatrix}
      1 & 1 & \cdots & 1
     \end{bmatrix} 
\mathrm{diag} [(\sigma^{-1}\xi)^p]
 \begin{bmatrix}
	z_{\sigma} 
\end{bmatrix}_{\sigma \in G} \\
\intertext{since $\begin{bmatrix}
      1 & 1 & \cdots & 1
     \end{bmatrix}$ is the first row of $P$, }
& = \text{the first element of }  P \mathrm{diag} [(\sigma^{-1}\xi)^p] P^{-1} \begin{bmatrix}
							       a_0 \\ \vdots \\
a_{d-1}       \end{bmatrix} && \text{by } \eqref{eq:co} \\
 &=\text{the first element of } R_F^{p}  \begin{bmatrix} 					       a_0 \\ \vdots \\
a_{d-1}       \end{bmatrix} \\
& \equiv  a_p \pmod{p} && \text{by } \eqref{eq:2.3} 
\end{align*}
and therefore, the sequence $\mathbf{a}$ satisfies \eqref{eq:1.2}.

The authors of \cite{MR4755038} constructed
such characteristic functions
as the idempotents corresponding 
to the decomposition $X (G,L)^G \cong \prod_{j=1}^r
L^{Z_G(\rho_j)}$ (see \cite[Proposition 2.11 and (3.2)]{MR4755038}). 

We instead construct explicit idempotents in $X(G, \Q_{\chi})$ where 
$\Q_{\chi}$ is an abelian extension of $\Q$ we will soon define.
Our strategy is as follows.
Such a function $\phi_K$ is obviously a class function and thus contained 
in the center $Y(G,\Q)$ of $X(G,\Q)$ with respect to the convolution product.
Let $\Q_{\chi}$ be the field extension of $\Q$ generated by
the values of the irreducible characters of $G$.
The representation theory of finite groups tells us that 
the function $\phi_K$ can be described explicitly, 
if we work on $\Q_{\chi}$ instead of $\Q$. 
Finally, we shall show that the inverse image $\mathbf{a}$ of $\phi_K$ exists.
This strategy will be accomplished in the next section.

\section{Proof of Theorem \ref{thm:1.1}} \label{sec:3}
We recall from Section \ref{sec:2} that $Y(G,\Q_{\chi})$ is the subalgebra 
consisting
of the class functions on $G$
with values in $\Q_{\chi}$.

In this section, we first give standard class functions in $Y(G,\Q_{\chi})$
that are essentially the characteristic functions on the  conjugacy classes of $G$.
By changing bases of $L \otimes_{\Q} Y (G,\Q_{\chi})$, we shall prove
Theorem \ref{thm:1.1} at the end of this section.

The following lemma is well known.

\begin{lemma}[\protect{\cite[XVIII Theorem 4.9]{MR1878556}}]
 The subalgebra $Y (G, \Q) $ is a natural dual of the center $Z (\Q [G])$.
\end{lemma}

The duality follows from the bilinear map $(f, \sigma ) = f (\sigma )$.
It is also well known that the conjugacy sums
\[
 k_j = \sum_{\sigma \in K_j } \sigma \in \Q [G] \quad (j=1,\ldots , r)
\]
form a basis of $Z (\Q [G])$ (see \cite[XVIII Proposition 4.2]{MR1878556}).

\begin{lemma} \label{lem:3.2}
Let 
\[
 k_j^* = \frac{1}{|Z_G  (\rho_j )|} \sum_{\psi \in \mathrm{Irr} (G)} 
\psi (\rho_j^{-1}) \psi \in Y (G, \Q_{\chi}).
\]
Then $k_j^*$ is the characteristic function of $K_j$, namely 
$k_j^* (\rho_i)= \delta_{ji}$ holds.
\end{lemma}
\begin{proof}
By definition, we have
\[
 k_j^* (\rho_i) = \frac{1}{|Z_G  (\rho_j )|} \sum_{\psi \in \mathrm{Irr} (G)} 
\psi (\rho_j^{-1}) \psi (\rho_i).
\]
The second orthogonality relation (\cite[Theorem 2.18]{MR0460423}) implies 
that the sum on the right hand side is equal to $0$ if $i \ne j$ and to 
$|Z_G (\rho_j)|$ otherwise.
\end{proof}

By the above proof, the equality 
$k_j^* (k_i) = |K_j| \delta_{ji}$ holds and this implies that
\[
 \frac{k_j^*}{|K_j|} = \frac{1}{|G|} \sum_{\psi \in \mathrm{Irr} (G)} 
\psi (\rho_j^{-1}) \psi  \quad   (j=1,\ldots , r)
\]
forms a dual basis of $(k_j)_{1 \le j \le r}$.

For later use, we define the matrix $C $ by
\[
 C= [\psi (\rho_j^{-1})]_{\psi \in \mathrm{Irr} (G), 1 \le j \le r}
\in \mathrm{Mat}_{r\times r} (\Q_{\chi})
\]
and then the identity
\begin{equation} \label{eq:3.1}
 \left( |Z_G (\rho_j)| \, k_j^* \right)_{1 \le j \le r} = (\psi )_{\psi \in \mathrm{Irr} (G)} C
\end{equation}
 holds by Lemma \ref{lem:3.2}.
The matrix $C$ is almost equal to the character table of $G$.

For $\alpha \in L$,
let $\varphi_{\alpha } \in X(G,L)$ be the map sending $\sigma \in G$ to 
$ \sigma \alpha $, which is the image of $1 \otimes \alpha $ by $\Phi $ in Lemma \ref{lem:2.3}. 
Another standard basis $\mathrm{Irr} (G)$ of $Y (G,\Q_{\chi})$ 
can be expressed as linear combinations of 
$\varphi_{\sigma^{-1} \xi} \, (\sigma \in G)$ in $X (G,L)$ with coefficients
in $L$.

\begin{lemma} \label{lem:3.3}
Let $\kappa \in \mathrm{Mat}_{d \times r} (\Q)$ be as defined 
in the statement of Theorem \ref{thm:1.1}.
The following assertions hold.
\begin{enumerate}
 \item \label{3.3.1} The set $(\varphi_{\sigma^{-1} \xi })_{\sigma \in G}$ forms a basis of $X(G,L)$.
 \item \label{3.3.2}
If we set $\Gamma_{\xi} = [\sigma \tau^{-1} \xi ]_{\sigma \in G, \tau \in G} \in \mathrm{Mat}_{d \times d}
 (L)$, then we have an equality
\[
 (\psi)_{\psi \in \mathrm{Irr} (G)} = (\varphi_{\sigma^{-1} \xi })_{\sigma \in G} \Gamma_{{\xi}}^{-1} \kappa C'
\]
with $C'=[\psi (\rho_i)]_{1\le i \le r, \psi \in \mathrm{Irr}(G)}$.
\end{enumerate}
\end{lemma}
\begin{proof}
The first assertion \ref{3.3.1} is clear.
For \ref{3.3.2},
let $M$ be a matrix satisfying 
$ (\psi)_{\psi \in \mathrm{Irr} (G)} = (\varphi_{\sigma^{-1} \xi })_{\sigma \in G} M$. By noting that the resulted function is a class function, 
it is readily seen that 
\[
\kappa C' = [ \varphi_{\sigma^{-1} \xi} (\tau )]_{ \tau \in G, \sigma \in G} 
\, M = \Gamma_{\xi} \, M.
\]
From this, the lemma follows immediately.
\end{proof}

Combining the above lemmas, we can prove the following lemma.

\begin{lemma} \label{lem:3.4}
 We have 
\[
 (k_j^*)_{1 \le j \le r} = (\varphi_{\sigma^{-1}\xi})_{\sigma \in G} \Gamma_{\xi}^{-1} \kappa .
\]
\end{lemma}
\begin{proof}
 Substituting the formula from Lemma \ref{lem:3.3} into \eqref{eq:3.1} yields
\[
 \left( |Z_G (\rho_j)| \, k_j^* \right)_{1 \le j \le r} = (\varphi_{\sigma^{-1}\xi})_{\sigma \in G} \Gamma_{\xi}^{-1} \kappa C' C.
\]
The $(i,j)$-entry of $C'C$ is
\[
 \sum_{\psi \in \mathrm{Irr} (G)} \psi (\rho_i) \psi (\rho_j^{-1}) = \begin{cases}
								      |Z_G(\rho_j)| & \text{ if } i=j , \\
								      0 & \text{ otherwise}
								     \end{cases}
\]
by the second orthogonality relation.
From this, the equality in the lemma follows.
\end{proof}

We are now ready to prove Theorem \ref{thm:1.1}.

\begin{proof}[Proof of Theorem \ref{thm:1.1}]
 Let $\mathbf{a}_{K_j}=(a_{K_j,i})_{ i \ge 0} \in \mathrm{RS} (F,L)$ be the sequence satisfying 
\eqref{eq:1.2} for the conjugacy class $K_j$.
They must be chosen so that $\Phi \circ \Psi (\mathbf{a}_{K_j}) = k_j^{*}$ 
holds.
If we collect the initial $d$ terms together in a matrix
 $A=[a_{K_j,i}]_{1 \le i \le d, 1 \le j \le r}$, then it follows from 
\eqref{eq:co}, \eqref{eq:2.2.5}, and Lemma \ref{lem:3.4} that
\[
 (\varphi_{\sigma^{-1} \xi})_{\sigma \in G} P^{-1} A = 
(k_j^*)_{1 \le j \le r} =  (\varphi_{\sigma^{-1} \xi})_{\sigma \in G} 
\Gamma_{\xi}^{-1} \kappa .
\]
Since $ (\varphi_{\sigma^{-1} \xi})_{\sigma \in G}$ is a basis of $X (G,L)$
by Lemma \ref{lem:3.3}, 
we conclude $A=P\, \Gamma_{\xi}^{-1} \kappa $.
 \end{proof}

Here we note that the result of Theorem \ref{thm:1.1} does not 
have a mark of the base change $\Q$ to $\Q_{\chi}$.

\section{Proof of Theorem \ref{thm:1.2}} \label{sec:4}
In this section, we compute the inverse of the matrix $\Gamma_{\xi}$ defined 
in Lemma \ref{lem:3.3} and 
prove the formulas in Theorem \ref{thm:1.2}.

We begin with the definition of the classical object from the origin of 
representation theory.

\begin{definition} \label{def:4.1}
 Let $\{ X_\sigma \}$ be variables indexed by $G$.
We call
\[
 \Gamma = \Gamma (G) = [X_{\sigma \tau^{-1}}]_{\sigma \in G, \tau \in G}
\]
the \emph{group matrix} of $G$.
The determinant of $\Gamma $ is called the \emph{group determinant} of $G$.
\end{definition}

The factorization of the group determinant was questioned by Dedekind
and answered by Frobenius. For this historical development, there is an 
excellent survey \cite{MR1659232} by Keith Conrad.

\begin{proposition}[\protect{\cite[Section 5]{MR1659232}}] \label{prop:4.1}
For each $\psi \in \mathrm{Irr} (G)$, let 
\begin{equation} \label{eq:4.1}
 \Delta (\psi ) 
= \det \left[
	\sum_{\sigma \in G} X_{\sigma } \Lambda_{\psi} (\sigma)
       \right].
\end{equation}
We then have a factorization of $\det \Gamma $ into irreducible polynomials:
\[
 \det \Gamma = \prod_{\psi \in \mathrm{Irr} (G)}  \Delta (\psi )^{\psi (1)} .
\]
 \end{proposition}

A Jacobi-type formula for an inverse matrix implies the following lemma.

\begin{lemma}[\protect{\cite[Lemma 1]{MR1659232}}] \label{lem:4.3}
If we define 
\[
 \partial_{\sigma } = \frac{1}{|G|} 
\frac{\partial \det \Gamma }{\partial X_{\sigma}},
\]
then the $(\sigma, \tau )$-entry of the adjugate matrix of $\Gamma $  is
given by $ \partial_{\tau \sigma^{-1}}$.
Namely, we have $(\det \Gamma ) \, E_{|G|} = \Gamma \, 
[\partial_{\tau \sigma^{-1}}]_{\sigma \in G, \tau \in G }$, where $E_{|G|}$ is the identity matrix of size $|G|$.
\end{lemma}

\begin{lemma} \label{lem:4.4}
 If we define the left action of $\gamma \in G$ on $X_{\tau }$ by
$\gamma X_{\tau} = X_{\gamma \tau}$, then the action of $\gamma $ on
$\partial_{\tau}$ is given by 
\[
 \gamma \partial_{\tau } = (\mathrm{sgn} \, \gamma ) \, \partial_{\gamma \tau} .
\]
\end{lemma}
\begin{proof}
We first note that, if we define a permutation matrix $S$ by $(\gamma \sigma )_{\sigma \in G}
= (\sigma)_{\sigma \in G} S$, 
then
\[
 \gamma \Gamma = [X_{\gamma \sigma \tau^{-1}}]_{\sigma \in G, \tau \in G} 
= [X_{\sigma \tau^{-1}}]_{\sigma \in \gamma G, \tau \in G}
= S \Gamma 
\]
for $\gamma \in G$.
It is plain to see that $\det S = \mathrm{sgn} \, \gamma  $.
We now compute
\[
 \gamma \partial_{\tau}  = \frac{1}{|G|} 
\frac{\partial \det (\gamma \Gamma)}{\partial X_{\gamma \tau}} 
 = \frac{1}{|G|} 
\frac{\partial \det (S \Gamma)}{\partial X_{\gamma \tau}} 
 =\frac{\mathrm{sgn} \, \gamma }{|G|} 
\frac{\partial \det  \Gamma }{\partial X_{\gamma \tau}}
=(\mathrm{sgn}\, \gamma )  \partial_{\gamma \tau}
\]
as desired.
\end{proof}

As stated in Theorem \ref{thm:1.2}, 
we consider $\Gamma_{\xi}, \Delta (\psi )_{\xi},$ and $\partial_{\sigma} (\xi)$
  as $\Gamma , \Delta_{\psi },$ and $\partial_{\sigma }$, 
respectively evaluated at $\xi$.

We now prove Theorem \ref{thm:1.2}.

\begin{proof}[Proof of Theorem \ref{thm:1.2}]
First, we recall from the proof of Theorem \ref{thm:1.1}
\[
 A= [a_{K_j,i}]_{1 \le i \le d, 1 \le j \le r }
\]
and $P$ as in \eqref{eq:2.0} and $\kappa $ as in Theorem \ref{thm:1.1}.
Theorem \ref{thm:1.1} then states that $A=P \Gamma_{\xi}^{-1} \kappa$.
It follows from Lemma \ref{lem:4.3} that the adjugate matrix 
$\widetilde{\Gamma}_{\xi}$ of $\Gamma_{\xi}$
equals to $[\partial_{\tau \sigma^{-1}} (\xi)]_{\sigma \in G, \tau \in G}$ and thus we obtain
 $(\det \Gamma_{\xi} )\, A = P \widetilde{\Gamma}_{\xi} \kappa$.
If $i \ge 0$, then by \eqref{eq:2.3}, the $i$-th term $a_{K_j, i}$ is the $(1,j)$-entry of the matrix
\begin{multline*}
 R_F^i A =P \mathrm{diag} [ (\sigma^{-1} \xi )^i]_{\sigma \in G} P^{-1 } A
= P \mathrm{diag} [(\sigma^{-1} \xi )^i]_{\sigma \in G} P^{-1 } P \Gamma_{\xi}^{-1} \kappa \\
= \frac{1}{\det \Gamma_{\xi}} P \mathrm{diag} [(\sigma^{-1} \xi^i )]_{\sigma \in G} [\partial_{\tau \sigma^{-1}} (\xi)]_{\sigma \in G, \tau \in G} \, \kappa .
\end{multline*}
The $(1,j)$-entry of the above matrix is equal to
\begin{align*}
  a_{K_j,i} & = [\sigma^{-1}\xi^i]_{\sigma \in G} 
 \prescript{t}{}{\left[\sum_{\tau \in K_j} \partial_{\tau \sigma^{-1}} (\xi )\right]_{\sigma \in G}} \\
& = \frac{1}{\det \Gamma_{\xi}}  \sum_{\sigma \in G } \sigma^{-1} \xi^i  \sum_{\tau \in K_j} 
\partial_{\tau \sigma^{-1}} (\xi ) \\
&= \frac{1}{\det \Gamma_{\xi}} \sum_{\sigma \in G } \sigma^{-1} \xi^i  \sum_{\tau \in K_j} 
\partial_{\sigma^{-1} \tau } (\xi ) & & (\because \tau \mapsto \sigma^{-1} \tau \sigma ) \\
&= \frac{1}{\det \Gamma_{\xi}} \sum_{\sigma \in G } \sigma \xi^i  \sum_{\tau \in K_j} 
\partial_{\sigma\tau } (\xi ) & & (\because \sigma  \mapsto \sigma^{-1} ) 
\end{align*}
This shows that the formula holds for $i \ge 0$.

It remains to show the latter half of the theorem.
First we assume $G \not\subset A_d$ and compute
\begin{align*}
 \det \Gamma_{\xi} \,  a_{K_j,i} & = 
\sum_{\sigma \in G } \sigma \xi^i   \sum_{\tau \in K_j} 
 \partial_{ \sigma \tau } (\xi ) \\
& = \sum_{\sigma \in G } \mathrm{sgn} (\sigma ) \sigma \left ( \xi^i    \sum_{\tau \in K_j} 
 \partial_{  \tau } (\xi ) \right) && (\text{by Lemma } \ref{lem:4.4}) \\
&= \sum_{\sigma \in G^+} \sigma 
\left(\xi^i \sum_{\tau \in K_j}  \partial_{\tau } (\xi ) \right)  - \delta \sum_{\sigma \in G^+} \sigma 
\left( \xi^i \sum_{\tau \in K_j}  \partial_{\tau } (\xi ) \right)
&& (\because G= G^{+} \sqcup \delta G^{+}) \\ 
&= (\mathrm{Tr}_{G^{+}} -\delta \mathrm{Tr}_{G^{+}}) \left(  \xi^{i}  \sum_{\tau \in K_j}  \partial_{\tau} (\xi ) \right).
\end{align*}
The case $G \subset A_d$ also follows from a similar and simpler computation.
\end{proof}

While, in their paper \cite{MR4755038}, 
the rationality of the sequences is proved indirectly using \eqref{eq:2.8},
 we can prove the rationality directly from Theorem \ref{thm:1.2}.

\begin{corollary} \label{cor:4.5}
 For any $K_j$ and $i \ge 0$, the term $a_{K_j,i}$ is a rational number.
\end{corollary}
\begin{proof}
 Let $\gamma \in G$. By the proof of Lemma \ref{lem:4.4}, we have 
$\det (\gamma \Gamma ) = \mathrm{sgn} (\gamma ) \det \Gamma $.
Hence, Lemma \ref{lem:4.4} implies 
\begin{align*}
 \gamma a_{K_j,i} & = \frac{1}{\mathrm{sgn} \gamma \det \Gamma_{\xi} }
\sum_{\substack{ \sigma \in G \\ \tau \in K_j }} \gamma \sigma \xi^i
\gamma (\partial_{  \sigma \tau } (\xi))\\
&= \frac{1}{\mathrm{sgn} \gamma \det \Gamma_{\xi} }
\sum_{\substack{ \sigma \in G \\ \tau \in K_j}} \gamma \sigma \xi^i
(\mathrm{sgn} \, \gamma ) \partial_{\gamma   \sigma \tau}   .
\end{align*}
If $\sigma $  runs through $G$, then so does $\gamma \sigma $ and hence,
 $\gamma a_{K_j,i} = a_{K_j,i}$ holds for all $\gamma \in G$.
\end{proof}

\section{Examples} \label{sec:5}
In this section, we give examples of our theorems.
In the following computation, we use Magma \cite{MR1484478}.

\begin{example}
In this example, we consider a sextic polynomial $F(X)$.
Assume that the Galois group of $F$
is isomorphic to $S_3$ and also that a root $\xi$ of $F(X)$ 
generates a normal basis of the splitting field $L$.
Moreover, we assume that $S_3 \subset S_6$ is ordered as 
\begin{multline*}
( 1,
   \tau= (1 \,  4)(2 \,  3)(5 \, 6),
    (1\,  2)(3\,  6)(4\,  5),     (1\,  6)(2\,  5)(3\,  4),    \\
   \sigma= (1\,  3\,  5)(2\,  4\,  6),
    (1\,  5\,  3)(2\,  6\,  4) ).
 \end{multline*}
As a transitive permutation group, the ID of the group is 6T2.
The conjugacy classes are
\[
 K_1 = \{1\}, \ K_2= \{ \tau , \tau \sigma  , \tau \sigma^2  \}, \ 
K_3  = \{ \sigma , \sigma^2 \} .
\]
Here we assume that $G$ acts on $\{ 1, \ldots , 6\}$ on the left.

We compute 
\[
 P=\begin{bmatrix}
1 & 1 & 1 & 1 & 1 & 1 \\
X_1 & X_{\tau} & X_{\tau \sigma } & X_{\tau \sigma^2} & X_{\sigma^2} & X_{\sigma} \\[3pt]
X_1^2 & X_{\tau}^2 & X_{\tau \sigma }^2 & X_{\tau \sigma^2}^2 & X_{\sigma^2}^2 & X_{\sigma}^2 \\[3pt]
X_1^3 & X_{\tau}^3 & X_{\tau \sigma }^3 & X_{\tau \sigma^2}^3 & X_{\sigma^2}^3 & X_{\sigma}^3 \\[3pt]
X_1^4 & X_{\tau}^4 & X_{\tau \sigma }^4 & X_{\tau \sigma^2}^4 & X_{\sigma^2}^4 & X_{\sigma}^4 \\[3pt]
X_1^5 & X_{\tau}^5 & X_{\tau \sigma }^5 & X_{\tau \sigma^2}^5 & X_{\sigma^2}^5 & X_{\sigma}^5 
\end{bmatrix} , \quad
\Gamma = 
\begin{bmatrix}
X_1 & X_{\tau} & X_{\tau \sigma } & X_{\tau \sigma^2} & X_{\sigma^2} & X_{\sigma} \\
X_{\tau} & X_1 & X_{\sigma} & X_{\sigma^2} & X_{\tau \sigma^2 } & X_{\tau \sigma} \\
X_{\tau \sigma } & X_{\sigma^2} & X_1 & X_{\sigma} & X_{\tau } & X_{\tau \sigma^2} \\
X_{\tau \sigma^2} & X_{\sigma} & X_{\sigma^2} & X_1 & X_{\tau \sigma} & X_{\tau  } \\
X_{\sigma} & X_{\tau \sigma^2 } & X_{\tau } & X_{\tau \sigma } & X_1 & X_{\sigma^2} \\
X_{\sigma^2} & X_{\tau \sigma} & X_{\tau \sigma^2} & X_{\tau  } & X_{\sigma} & X_1 
\end{bmatrix}.
\]
Since $\mathrm{Irr} (G) = \{ 1, \chi, \psi \}$ with
 nontrivial characters $\chi, \psi $ satisfying
$\chi (1)=1, \psi (1)=2$,
 we have
\begin{align*}
 \Delta (1)  =  & X_1 + X_{\tau} + X_{\tau \sigma } + X_{\tau \sigma^2} + X_{\sigma} + X_{\sigma^2} , \\
 \Delta (\chi ) =& X_1 - X_{\tau} - X_{\tau \sigma } - X_{\tau \sigma^2} + X_{\sigma} + X_{\sigma^2} , \\
 \Delta (\psi ) = & X_1^2 - X_1 X_{\sigma} - X_1 X_{\sigma^2} - X_{\tau}^2 + X_{\tau} X_{\tau \sigma } \\ 
      & + X_{\tau} X_{\tau \sigma^2} - X_{\tau \sigma }^2 + 
        X_{\tau \sigma } X_{\tau \sigma^2} - X_{\tau \sigma^2}^2 + X_{\sigma}^2 - X_{\sigma} X_{\sigma^2} + X_{\sigma^2}^2 
\end{align*}
and thus, by Proposition \ref{prop:4.1}, 
\[
 \det \Gamma = \Delta (1) \Delta(\chi) \Delta (\psi )^2
\]
holds.
We also compute 
\begin{multline*}
 \partial_{1} = \frac{1}{6} \Delta (\psi ) \,
\left( X_1^3 + X_1^2 X_{\sigma} + X_1^2 X_{\sigma^2} - X_1 X_{\tau}^2 - X_1 X_{\tau} X_{\tau \sigma } - 
        X_1 X_{\tau} X_{\tau \sigma^2} - X_1 X_{\tau \sigma }^2 \right. \\
- X_1 X_{\tau \sigma } X_{\tau \sigma^2}   - X_1 X_{\tau \sigma^2}^2 - 
        X_1 X_{\sigma} X_{\sigma^2} + X_{\tau} X_{\tau \sigma } X_{\sigma} + X_{\tau} X_{\tau \sigma } X_{\sigma^2} \\
\left. + X_{\tau} X_{\tau \sigma^2} X_{\sigma}  + 
        X_{\tau} X_{\tau \sigma^2} X_{\sigma^2}  + X_{\tau \sigma } X_{\tau \sigma^2} X_{\sigma} + X_{\tau \sigma } X_{\tau \sigma^2} X_{\sigma^2} - X_{\sigma}^2 X_{\sigma^2} - 
        X_{\sigma} X_{\sigma^2}^2 \right).
\end{multline*}
Other partial derivatives can be computed by using Lemma \ref{lem:4.4}. 
Since $G^{+} = \langle \sigma \rangle, $
Theorem \ref{thm:1.2} yields that 
\begin{align*}
 a_{K_1,i} & = \frac{1}{ \det \Gamma_{\xi}} (\mathrm{Tr}_{G^{+}} -\delta \mathrm{Tr}_{G^{+}}) \left(  \xi^{i}  \partial_{1} (\xi)  \right), \\
  a_{K_2,i} & = \frac{1}{ \det \Gamma_{\xi}} (\mathrm{Tr}_{G^{+}} -\delta \mathrm{Tr}_{G^{+}}) \left(  \xi^{i}  ( \partial_{\tau} + \partial_{\tau\sigma } + \partial_{\tau\sigma^2} ) (\xi)  \right), \\
  a_{K_3,i} & = \frac{1}{ \det \Gamma_{\xi}} (\mathrm{Tr}_{G^{+}} -\delta \mathrm{Tr}_{G^{+}}) \left(  \xi^{i}  ( \partial_{\sigma} + \partial_{\sigma^2 }) (\xi )    \right)
\end{align*}
for all $i \ge 0$.

See \cite[Proposition 1.5]{MR4755038} for a numerical example.
\end{example}

\begin{example}
 In this example, we consider the polynomial
\[
 G(X)=X^{12} + 8 X^{10} - 7X^8 + 266X^6 + 4X^4 + 1856X^2 + 64.
\]
The Galois group of $G(X)$ is isomorphic to $A_4$ considered as a subgroup of $S_{12}$.
As a transitive permutation group, the ID of the group is 12T4.
The conjugacy classes $K_i \, (1 \le i \le 4)$ are represented by
\begin{align*}
&\rho_1 = 1, && \rho_2 =  (1 \, 7)(2 \, 8)(3 \, 9)(4 \, 10)(5 \, 11)(6 \ 12), \\
 & \rho_3= (1\, 9\, 5)(2\, 4\, 3)(6\, 8\, 7)(10\, 12\, 11), & & \rho_4 = (1\, 5\, 9)(2\, 3\, 4)(6\, 7\, 8)(10\, 11\, 12).
\end{align*}
If $a$ is a root of $F(X)$, then
\[
 \xi = \frac{1}{8}(-a^9 - 2 a^7 - 3 a^5 + 4 a^4 - 2 a^3 + 8 a^2 + 4 a - 8)
\]
generates a normal basis\footnote{Here we randomly choose an element $\xi$ and check whether $\det \Gamma_{\xi} = 0$ or not.} of the splitting field $L$ 
of $G(X)$.
The minimal polynomial of $\xi$ is 
\begin{multline} \label{eq:5.1}
 F(X) = X^{12} - 50 X^{11} + 25493717 X^{10} + 1180522454 X^9 + 885646536828 X^8 \\ 
+     32213365427586 X^7 + 7576066505433801 X^6   \\
+ 172513732994029238 X^5 +    1238519020079163369 X^4 \\ + 8381731881286962404 X^3 +
    137777496999186460164 X^2 \\ + 262600792236322628544 X + 
    133318780807273404444.
\end{multline}
We compute the initial $d$ terms for $\xi$ by Theorem \ref{thm:1.1} and obtain 
for $K_1$
\[
\begin{bmatrix}
1/50 \\
1 \\
-4135364071611241/14424104400 \\
-260547744288653071/4808034800 \\
26300466432052004553821/3606026100 \\
500827742889717806352169/240401740 \\
-668975249439161245987029364843/3606026100 \\
-4263502226999619804770492844999/60100435 \\
4252431954040670758350722151769286953/901506525 \\
2718092977698862953735451274910193880983/1202008700 \\
-108084862311333479116432375783494868959858356/901506525 \\
-249332171693508539722479652354516785602310989741/3606026100 
\end{bmatrix},
\]
and for $K_2$
\[
 \begin{bmatrix}
3/50 \\
0 \\
-4964786092726823/14424104400 \\
-178110438425349713/4808034800 \\
31574757808051685512063/3606026100 \\
429793472896875575390307/240401740 \\
-803379315883723829261636004029/3606026100 \\
-4027348355322179825440668139282/60100435 \\
5108370449133905137640756221859719934/901506525 \\
2707893528532495505903540005737181936049/1202008700 \\
-129880723115465709986373470186715277529010493/901506525 \\
-256961308855093814076389357839704579900571992223/3606026100 
\end{bmatrix},
\]
and for $K_3$
\[
 \begin{bmatrix}
2/25 \\
0 \\
-379245606/1975 \\
-54786318908/1975 \\
9638668421452394/1975 \\
464962850044199918/395 \\
-245209041778488707648152/1975 \\
-16569309380188555032415002/395 \\
6235871278796055129644393225618/1975 \\
2710965787654331821079588907831936/1975 \\
-158525166496531958209545246760654609686/1975 \\
-84319186486699737581686554509272598366474/1975 
\end{bmatrix},
\]
and finally for $K_4$
\[
 \begin{bmatrix}
2/25 \\
0 \\
-388634156/1975 \\
-55962982458/1975 \\
9848511058081944/1975 \\
475084503041337148/395 \\
-250546971352416621277002/1975 \\
-16930004890600853013935032/395 \\
6371619289231042593431953974568/1975 \\
2769980509985893459404994437740986/1975 \\
-161976083777001208967454900637719599036/1975 \\
-86154721778368501577407226916799238923024/1975 
\end{bmatrix}.
\]
Multiplying the companion matrix $R_F$ of \eqref{eq:5.1} to
the above vectors to compute $a_p$,
 we can derive the following classification from Theorem \ref{thm:1.1} :
\begin{align*}
 & \mathrm{Frob}_{p} \in K_1 \iff && p \in \{ 97 , 173, 269 \}, \\
 & \mathrm{Frob}_{p} \in K_2 \iff && p \in \{ 41, 47, 59,113,127,131,151,193, 199,
211, 223,257,293 \}, \\
 & \mathrm{Frob}_{p} \in K_3 \iff && p \in \{19,23,31,37,53,67,83,89,109,139,149,
167,191,227, 
\\ && & 239,263,277,281 \}, && \\
 & \mathrm{Frob}_{p} \in K_4  \iff && p \in \{7,13,17,29,61,71,73,101,103,157,163,
179,181,197, && \\ & && 229,233,241,251,271,283 \}
\end{align*}
for primes up to $300$.
Note that $|K_1|=1,|K_2|=3, \text{ and } |K_3|=|K_4|=4 $.

There are several exceptional primes for each sequence even up to $300$
including $ 2, 3, 5, 11, 43, 107, 137$ dividing the discriminant of $F$.

\end{example}

\section{Concluding remark} \label{sec:6}
As discussed in \cite[Section 4]{MR4755038},
it is an important problem to determine the exceptional finite set $S$ of
primes for \eqref{eq:1.2}. The set $S$ obviously contains the ramified primes 
in $L/\Q$ and also primes dividing the denominators of 
the $d$ initial terms, which are divisors of $\det \Gamma_{\xi}$.

This section briefly discusses the relationship between this problem and
the integral model of a certain algebraic torus.

We consider the $d$-dimensional torus $ T=\qst{L/\Q}$ defined over $\Q$.
If we take a power basis $(1,\xi, \ldots , \xi^{d-1})$ of $L$, then
we can identify the generic point of $T$ as a matrix of 
regular representation of $L^{\times}$, which is a polynomial 
of $R_F$ with rational coefficients.
The $d$-initial terms of the sequence can be considered as a rational point in
$T(\Q )$. Since $T (\Q) \cong L^{\times}$, we have an element of $L$
(see $\Lambda $ in Lemma \ref{lem:2.1}). Since $T$ splits over $L$, we have \eqref{eq:2.0.5}
and also the duality between the character group $\widehat{T}$ (see Lemma 
\ref{lem:2.3}).
Therefore, the determination of the exceptional set $S$ is closely related 
to choosing a good integral model of $T$.
In \cite[Sections 10.4 and 11.2]{MR99g:20090}, 
a canonical integral model of a torus
over global fields is proposed and is shown that the model has very good 
reduction, meaning that the mod $p$ reduction is also a torus, outside primes
dividing the discriminant of $L$. 
Although we also have to find a model whose root generates a normal basis,
the study of integral models of the torus could be a clue to determining $S$.
In the special case where $[L:\Q ]=2$, this problem is carefully treated in 
\cite{MR3902117} in different language.

\subsection*{Acknowledgments}
The authors wish to express their thanks to the anonymous referees for
their helpful comments that improved the quality of the manuscript.
The authors also thank Genki Koda, Yuta Katayama, and Nozomu Suzuki for useful discussions 
during the preparation of this paper.

\subsection*{Ethics approval}
Not applicable.


\providecommand{\bysame}{\leavevmode\hbox to3em{\hrulefill}\thinspace}

 \begin{flushleft}
Haruto Hori \\ 
 e-mail: 1122523@ed.tus.ac.jp \\
 Masanari Kida \\
 e-mail: kida@rs.tus.ac.jp \\
 Department of Mathematics \\
 Tokyo University of Science \\
 1-3 Kagurazaka Shinjuku Tokyo 162-8601 Japan 
 \end{flushleft}

\end{document}